\documentclass[11pt]{article}
\usepackage[utf8]{inputenc}
\usepackage[OT1]{fontenc}
\usepackage{amsmath,amssymb,amsthm,bm}
\usepackage{thmtools}
\usepackage{thm-restate}
\usepackage[noblocks,affil-it]{authblk}
\usepackage[mathscr]{euscript}
\usepackage[margin=3cm]{geometry}
\usepackage{hyperref}
\usepackage{todonotes}

\newcommand{\R}{{\mathbb{R}}}
\newcommand{\Z}{{\mathbb{Z}}}
\newcommand{\bb}{{\mathscr{B}}}
\newcommand{\aaa}{{\mathscr{A}}}
\newcommand{\eps}{{\varepsilon}}
\newcommand{\zero}{\mathbf{0}}
\newcommand{\la}{\langle}
\newcommand{\ra}{\rangle}

\declaretheorem{theorem}
\declaretheorem{lemma}

\title{\textbf{Binary scalar products}}
\author[1]{Andrey Kupavskii}
\affil[1]{Moscow Institute of Physics and Technology, Russia}
\affil[1]{Institute for Advanced Study, Princeton, USA}
\affil[1]{G-SCOP, CNRS, Grenoble, France}
\author[2]{Stefan Weltge}
\affil[2]{Technical University of Munich, Germany}
\date{}

\begin{document}

\maketitle

\begin{abstract}
    Let $\aaa,\bb \subseteq \R^d $ both span $\R^d$ such that $\la a, b\ra \in \{0,1\}$ holds for all $a \in \aaa$, $b \in \bb$.
    We show that~$ |\aaa| \cdot |\bb| \le (d+1) 2^d $.
    This allows us to settle a conjecture by Bohn, Faenza, Fiorini, Fisikopoulos, Macchia, and Pashkovich (2015) concerning 2-level polytopes.
    Such polytopes have the property that for every facet-defining hyperplane $H$ there is a parallel hyperplane $H'$ such that $H \cup H'$ contain all vertices.
    The authors conjectured that for every $d$-dimensional 2-level polytope $P$ the product of the number of vertices of $P$ and the number of facets of $P$ is at most~$d 2^{d+1}$, which we show to be true.
\end{abstract}

\section{Introduction}

For two vectors $a = (a_1,\ldots, a_d), b = (b_1,\ldots,b_d)$, let $\la a, b\ra = \sum_{i=1}^d a_ib_i$ their scalar product. Given two sets $\aaa,\bb \subseteq \R^d$ that both linearly span $\R^d$ with the property that $\la a, b\ra \in \{0,1\}$ holds for all $a \in \aaa$, $b \in \bb$, how many points can $\aaa$ and $\bb$ contain?
It is easy to see that each individual set cannot contain more than $2^d$ points.
This bound is tight since we may choose $\aaa = \{\zero, e_1,\dots,e_d\}$ and $\bb = \{0,1\}^d$.
However, it turns out that $|\aaa|$ and $|\bb|$ cannot be close to this bound simultaneously.
In fact, in this paper we prove the following.

\begin{theorem}
    \label{thmbinary}
    Let $\aaa,\bb \subseteq \R^d$ both linearly span $\R^d$ such that $\la a, b\ra \in \{0,1\}$ holds for all $a \in \aaa$, $b \in \bb$.
    Then we have $|\aaa| \cdot |\bb| \le (d+1) 2^d$.
\end{theorem}

The previous example also shows that this bound is tight. We note that if one restricts her attention to the families of vectors coming from  the Boolean cube $\{0,1\}^d$ then questions of similar nature are studied in extremal set theory. In particular, see \cite[Chapter 10]{frankl_forbidden_1987}. Certain extremal set theory-type problems for families of vectors coming from $\{0,\pm 1\}^d$ were studied in \cite{cherkashin_independence_2019, deza_bounds_1985, frankl_erdoskorado_2018, frankl_families_2018, frankl_intersection_2020}.

Our main motivation for studying this question is its close relation to point configurations associated to 2-level polytopes.
A polytope $P$ is said to be \emph{2-level} if for every facet-defining hyperplane $H$ there is a parallel hyperplane $H'$ such that $H \cup H'$ contains all vertices of $P$.
Basic examples of 2-level polytopes are hypercubes, cross-polytopes, and simplices.
Actually, 2-level polytopes generalize a variety of interesting polytopes such as Birkhoff, Hanner, and Hansen polytopes, order polytopes and chain polytopes of posets, stable matching polytopes, and stable set polytopes of perfect graphs~\cite{AprileCF}.
Moreover, they arise in different areas of mathematics, most notably in the field of extended formulations, which has received much attention during the past decade.

A fundamental result in polyhedral combinatorics states that $d$-dimensional stable set polytopes of perfect graphs admit subexponential (in $d$) size linear extended formulations, i.e., they are linear images of polytopes with subexponentially many facets~\cite{Yannakakis}.
It is a major open problem whether such polytopes have polynomial-size extended formulations.
Moreover, the famous log-rank conjecture by Lov{\'a}sz and Saks~\cite{LogRank} in the field of communication complexity would imply subexponential-size extended formulations for all 2-level polytopes.
However, no non-trivial bound is known for this general case.
In contrast, it is known that, among all $d$-dimensional polytopes, 2-level polytopes admit smallest possible semidefinite extended formulations~\cite{Gouveia}.
Details on these connections and several recent studies on 2-level polytopes can be found in~\cite{AprileCF,AprileFFHM,BohnESA,BohnMPC,Fiorini,Gouveia,Grande1,Grande2}.

Among them are extensive experimental studies by Bohn, Faenza, Fiorini, Fisikopoulos, Macchia, and Pashkovich~\cite{BohnESA,Fiorini,BohnMPC}, which led to a beautiful conjecture about the combinatorial structure of 2-level polytopes.
It is easy to see that for a $d$-dimensional 2-level polytope $P$ the number of vertices $f_0(P)$ and the number of facets $f_{d-1}(P)$ are both bounded by $2^d$.
Similar to the setting of Theorem~\ref{thmbinary}, Bohn et al.\ observed that for small values of $d$, $f_0(P)$ and $f_{d-1}(P)$ cannot be close to this bound simultaneously.
More specifically, they verified that every 2-level polytope of dimension $d \le 7$ satisfies $f_0(P) f_{d-1}(P) \le d 2^{d+1}$, which was later also confirmed for $d=8$ in~\cite{MacchiaPhD}.
In~\cite{BohnESA} it is asked whether this holds for all $d$, and in the journal version~\cite{BohnMPC} this is posed as a conjecture.

Recently, Aprile, Cevallos, and Faenza~\cite{AprileCF,AprilePhD} showed that this conjecture holds for many families of 2-polytopes, including the ones mentioned above.
We show that it is true for all 2-level polytopes:

\begin{restatable}{theorem}{twolevelresult}
    \label{thm2level}
    Every $d$-dimensional $2$-level polytope $P$ satisfies $f_0(P) f_{d-1}(P) \le d 2^{d+1}$.
\end{restatable}

Note that this bound is tight by choosing $P = [0,1]^d$.
Another simple consequence of Theorem~\ref{thmbinary} is the following.
Let $V$ be a finite set and let $\aaa$, $\bb$ be families of subsets of $V$ such that $|A \cap B| \le 1$ holds for all $A \in \aaa$, $B \in \bb$.
Then we have $|\aaa| \cdot |\bb| \le (|V|+1)2^{|V|}$.
For instance, this implies that for any $n$-node graph the number of its stable sets times the number of its cliques is bounded by $(n+1)2^n$ (see also \cite[Thm. 3.1]{AprileCF}), which is attained for the empty and complete graph, respectively.

\paragraph{Outline}
The proof of Theorem~\ref{thmbinary} is presented in the next section, in which we make use of several  claims, whose proofs are given in Section~\ref{secClaims}.
In Section~\ref{sec2level} we provide the proof of Theorem~\ref{thm2level}.

\section{Main proof}
\label{secMainProof}

In this section, we provide a proof for Theorem~\ref{thmbinary}. In what follows, depending on the situation, we treat the sets $\aaa,\bb$ either as sets of vectors or as sets of points. In particular, `span' always stands for `linearly span', while `$\dim $' stands for affine dimension. 

Let $f(d)$ be the maximum of $|\aaa| \cdot |\bb|$ over all $\aaa,\bb \subseteq \R^d$ that both span $\R^d$ such that $\la a, b\ra \in \{0,1\}$ for all $a \in \aaa$, $b \in \bb$.
We show that $f(d) \le (d+1)2^d$ holds by induction on $d \ge 0$.
Note that $f(0) = 1$ and let $d \ge 1$.

Let $\aaa,\bb \subseteq \R^d$ both span $\R^d$ such that $\la a, b\ra \in \{0,1\}$ for all $a \in \aaa$, $b \in \bb$.
We may assume that $\aaa$ and $\bb \setminus \{\zero\}$ are inclusion-wise maximal with respect to this property.\footnote{Clearly, we can always include $\zero$ in $\bb$. However, we emphasize that our proof only uses the inclusion-wise maximality of $\bb\setminus\{\mathbf 0\}$, a detail that becomes relevant for our application to 2-level polytopes.}
Note that every nonzero vector $x \in \R^d$ defines two faces of the convex hull of $\aaa$, and let $\varphi(x)$ denote the maximum of the dimensions of these two faces.
Let us pick $b_d \in \bb \setminus \{\zero\}$ such that $\varphi(b_d) \ge \varphi(b)$ holds for every $b \in \bb \setminus \{\zero\}$.

In what follows, we will invoke the induction hypothesis in the affine hull of one of the two faces that $b_d$ defines.
To this end, it will be convenient to slightly modify $\aaa$ and $\bb$, which is done in the following claim.
We say that a set $X \subseteq \R^d$ \emph{does not contain opposite points} if $|X \cap \{x,-x\}| \le 1$ holds for all $x \in \R^d$.
Let $$U := \{ x \in \R^d : \langle x,b_d \rangle = 0 \}$$ and consider the orthogonal projection $\pi : \R^d \to U$ on $U$.

\begin{restatable}{claim}{claimassumptions}
    \label{firstClaim}
    We may translate $\aaa$ and replace some points in $\bb$ by their negatives such that the following holds.
    \begin{itemize}
        \item[(i)] We can write $\aaa = \aaa_0 \cup \aaa_1$, where $\aaa_i = \{ a \in \aaa : \la a, b_d\ra = i\}$ for $i=0,1$ such that
            \begin{equation} \label{eqa0gea1}
                |\aaa_0| \ge |\aaa_1|.
            \end{equation}
        \item[(ii)] We still have
            \begin{equation}
                \label{eqscala0}
                \la a, b\ra \in \{0,1\} \text{ for each } a \in \aaa_0 \text{ and } b \in \bb.
            \end{equation}
        \item[(iii)] The set $\pi(\bb)$ does not contain opposite points.
    \end{itemize}
\end{restatable}

\noindent
Note that, after this transformation we may (still) assume that $\aaa$ contains $\zero$, and thus the transformation does not affect the property that $\aaa$ and $\bb$ both span $\R^d$. Also, this transformation does neither affect the choice of $b_d$, nor the cardinalities of $\aaa,\bb$.

\begin{restatable}{claim}{claimpreimagespi}
    \label{clapreimagespi}
    Every point in $\pi(\bb)$ has at most two preimages in $\bb$.
\end{restatable}

\noindent
Let $\bb_* \subseteq \bb$ denote the set of $b \in \bb$ for which $\pi(b)$ has a unique preimage. Claim~\ref{clapreimagespi} yields
$ |\bb \setminus \bb_*| = 2 |\pi(\bb \setminus \bb_*)| $.
We obtain
\begin{align}
    \nonumber
    |\aaa| |\bb| & = |\aaa_0||\bb \setminus \bb_*|  + |\aaa_0||\bb_*| + |\aaa_1||\bb \setminus \bb_*| + |\aaa_1||\bb_*| \\
    \nonumber
        & \le |\aaa_0||\bb \setminus \bb_*| + 2|\aaa_0||\bb_*| + |\aaa_1||\bb \setminus \bb_*| \\
    \nonumber
        & = 2 |\aaa_0|(|\pi(\bb \setminus \bb_*)| + |\bb_*|) + |\aaa_1||\bb \setminus \bb_*| \\
    \label{eq9930a}
        & = 2 |\aaa_0||\pi(\bb)| + |\aaa_1||\bb \setminus \bb_*|.
\end{align}
where the first inequality follows from~\eqref{eqa0gea1} and the last one from the definition of $\bb_*$.
We will bound the latter two terms separately.

Let us first provide a bound on the term $|\aaa_0||\pi(\bb)|$.
To this end, let $U_0 \subseteq U$ denote the subspace spanned by $\aaa_0$, and let $\tau : U \to U_0$ be the orthogonal projection onto $U_0$.
Note that $\tau(\pi(\bb))$ spans $U_0$ as $\bb$ spans $\R^d$.
Moreover, for each $a \in \aaa_0$ and each $b \in \bb$ we have
\[
    \la a, \tau(\pi(b))\ra = \la a, \pi(b)\ra = \la a, b\ra \in \{0,1\},
\]
where the last equality is due to~\eqref{eqscala0}.
Thus, we obtain $|\aaa_0||\tau(\pi(\bb))| \le f(\dim U_0)$ and hence the induction hypothesis yields
\begin{equation}
    \label{eq7731a}
    |\aaa_0||\tau(\pi(\bb))| \le (\dim U_0 + 1)2^{\dim U_0}.
\end{equation}
Moreover, we have the following relation between the sizes of $\tau(\pi(\bb))$ and $\pi(\bb)$:

\begin{restatable}{claim}{claimpreimagestau}
    We have $|\pi(\bb)| \le 2^{d - 1 - \dim U_0} |\tau(\pi(\bb))|$.
\end{restatable}

\noindent
Combining~\eqref{eq9930a} and~\eqref{eq7731a} with the above claim we thus obtain
\begin{equation}
    \label{eq1346a}
    |\aaa| |\bb| \le 2^{d - \dim U_0} |\aaa_0||\tau(\pi(\bb))| + |\aaa_1||\bb \setminus \bb_*|
        \le (\dim U_0 + 1)2^d + |\aaa_1||\bb \setminus \bb_*|.
\end{equation}
In order to bound the second term $|\aaa_1||\bb \setminus \bb_*|$, the following observation is useful.

\begin{restatable}{claim}{claimrestbbconstant}
    For each $b \in \bb \setminus \bb_*$ we have $|\{ \la a, b\ra : a \in \aaa_0 \}| = 1$ or $|\{ \la a, b\ra : a \in \aaa_1 \}| = 1$.
\end{restatable}

\noindent
The above claim implies that we can partition $\bb \setminus \bb_*$ into two sets $\bb_0,\bb_1$ where
\begin{align*}
    |\{ \la a, b\ra : a \in \aaa_0 \}| & = 1 \text{ for all } b \in \bb_0, \\
    |\{ \la a, b\ra : a \in \aaa_1 \}| & = 1 \text{ for all } b \in \bb_1.
\end{align*}
Note that we may choose this partition such that $\zero, b_d \notin \bb_0$.
By~\eqref{eq1346a} and~\eqref{eqa0gea1} we have
\begin{align}
    \nonumber
    |\aaa| |\bb| & \le (\dim U_0 + 1)2^d + |\aaa_1||\bb_0| + |\aaa_1||\bb_1| \\
    \label{eq8226}
        & \le (\dim U_0 + 1)2^d + |\aaa_0||\bb_0| + |\aaa_1||\bb_1|.
\end{align}
Finally, we make use of the following observation.

\begin{restatable}{claim}{claimaaaibbi}\label{claim5}
    For $i=0,1$ we have $|\aaa_i||\bb_i| \le 2^d$.
\end{restatable}

\noindent
If $\dim U_0 \le d-2$, then~\eqref{eq8226} together with the above claim yields
\[
    |\aaa| |\bb| \le (d - 1)2^d + 2 \cdot 2^d = (d+1) 2^d,
\]
as required.
It remains to consider the case $\dim U_0 = d - 1$.
In this case, the only nonzero  point in $\bb$ that has constant scalar product with all points in $\aaa_0$ is $b_d$.
Since $\zero,b_d \notin \bb_0$, we have $\bb_0 = \emptyset$.
Again by~\eqref{eq8226} and Claim~\ref{claim5} we conclude
\[
    \pushQED{\qed} 
    |\aaa| |\bb| \le d 2^d + |\aaa_1||\bb_1| \le d 2^d + 2^d = (d+1) 2^d.
    \qedhere \popQED
\]

\section{Proofs of claims}
\label{secClaims}

In this section, we provide the proofs of all previous claims, whose statements we repeat here.
We begin by explaining how the initial transformation in the proof of Theorem~\ref{thmbinary} can be performed.

\claimassumptions*
\begin{proof}
    If $|\{ a \in \aaa : \la a, b_d\ra = 0\}| \le |\{ a \in \aaa : \la a, b_d\ra = 1\}|$, then we can choose any $a_* \in \aaa$ with $\la a_*, b_d\ra = 1$ (which exists since $\aaa$ spans $\R^d$) and replace $\aaa$ by $\aaa - a_*$, $\bb$ by $(\bb \setminus \{b_d\}) \cup \{-b_d\}$, and $b_d$ by $-b_d$.
    This yields~(i).

    After this replacement, for each $b \in \bb$ there is some $\eps_b \in \{\pm 1\}$ such that $\la a, b\ra \in \{0,\eps_b\}$ holds for all $a \in \aaa$.
    Each $b$ with $\{\la a,b\ra : a \in \aaa_0\} = \{0,-1\}$ is replaced by $-b$, which yields~(ii).

    Let $\aaa_1'$ be a translate of $\aaa_1$ such that $\zero \in \aaa_1'$.
    Note that, for each $b \in \bb$ we now have $\{\la a,b\ra : a \in \aaa_0\} = \{0,1\}$ or $\{\la a,b\ra : a \in \aaa_0\} = \{0\}$.
    In the second case, we replace $b$ by $-b$ if $\{\la a,b\ra : a \in \aaa_1'\} = \{0,-1\}$, otherwise we leave it as it is.

    It remains to show that $\pi(\bb)$ does not contain opposite points after this transformation.
    To this end, let $b,b' \in \bb$ such that $\pi(b) = \beta \pi(b')$ for some $\beta \ne 0$, where $\pi(b),\pi(b') \ne \zero$.
    We have to show that $\beta = 1$.
    Note that for every $a \in \aaa_0 \cup \aaa_1' \subseteq U$ we have
    \[
        \la a,b\ra = \la a, \pi(b)\ra = \beta \la a, \pi(b')\ra = \beta \la a,b'\ra.
    \]
    Suppose first that $\{\la a,b\ra : a \in \aaa_0\} \ne \{0\}$.
    By~\eqref{eqscala0} there exists some $a \in \aaa_0$ with $1 = \la a,b\ra = \beta \la a,b'\ra$.
    Thus, we have $\la a,b'\ra \ne 0$ and hence $\la a,b'\ra = 1$, again by~\eqref{eqscala0}.
    This yields $\beta = 1$.

    Suppose now that $\{\la a,b\ra : a \in \aaa_0\} = \{0\}$.
    Note that this implies $\{\la a,b'\ra : a \in \aaa_0\} = \{0\}$.
    As $\aaa_0 \cup \aaa_1'$ spans $U$, we must have $ \{\la a,b\ra : a \in \aaa_1'\} \ne \{0\}$ and hence there is some $a \in \aaa_1'$ with $\la a,b\ra = 1$.
    Moreover, we have $\beta \la a,b'\ra = 1$, and in particular $\la a,b'\ra \ne 0$.
    This implies $\la a,b'\ra = 1$ and hence $\beta = 1$.
\end{proof}

As in the previous proof, let $\aaa_1'$ be a translate of $\aaa_1$ such that $\zero \in \aaa_1'$.
Note that for each $b \in \bb$ there are $\eps_b,\gamma_b \in \{\pm 1\}$ such that
\begin{align}
    \label{eqscaleps}
    & \la a, b\ra \in \{0,\eps_b\} \text{ for each } a \in \aaa \text{ and} \\
    \label{eqscala1pgamma}
    & \la a, b\ra \in \{0,\gamma_b\} \text{ for each } a \in \aaa_1'.
\end{align}

The proofs of the subsequent claims rely on the following two lemmas.

\begin{lemma}
    \label{lemslice}
    Suppose that $X \subseteq \{0,1\}^d \cup \{0,-1\}^d$ does not contain opposite points.
    Then we have $|X| \le 2^{\dim X}$.
\end{lemma}
\begin{proof}
    We prove the statement by induction on $d \ge 1$, and observe that it is true for $d = 1$.
    Now let $d \ge 2$.
    If $\dim X = d$, then we are also done.
    It remains to consider to case where $X$ is contained in an affine hyperplane $H \subseteq \R^d$.
    Let $c = (c_1,\ldots,c_d) \in \R^d$, $\delta \in \{0,1\}$ such that $$H = \{ x \in \R^d : \la c,x\ra = \delta \}.$$
    For each $i \in \{1,\dots,d\}$ let $\pi_i : H \to \R^{d-1}$ denote the projection that forgets the $i$-th coordinate, and let $e_i \in \R^d$ denote the $i$-th standard unit vector. Note that $\pi_{i^*}(X) \subseteq \{0,1\}^{d-1} \cup \{0,-1\}^{d-1}$.

    Suppose there is some $i^* \in \{1,\dots,d\}$ such that $\la c, e_{i^*}\ra \ne 0$ and $\pi_{i^*}(X)$ does not contain opposite points.
    By the induction hypothesis we obtain
    \[
        |X| = |\pi_{i^*}(X)| \le 2^{\dim \pi_{i^*}(X)} = 2^{\dim X},
    \]
    where the first equality and the last inequality hold since $\pi_{i^*}$ is injective (due to $\la c, e_{i^*}\ra \ne 0$).

    It remains to consider the case in which there is no such $i^*$.
    Consider any $i \in \{1,\dots,d\}$.
    If $\la c, e_i \ra \ne 0$, then there exist $x=(x_1,\ldots,x_d),x'=(x_1',\ldots,x_d') \in X$, $x \ne x'$ such that $\pi_i(x) = -\pi_i(x')$.
    We may assume that $\pi_i(x) \in \{0,1\}^{d-1}$ and hence $\pi_i(x') \in \{0,-1\}^{d-1}$.
    As $X$ does not contain opposite points, we must have $x_i = 1$ and $x'_i = 0$, or $x_i = 0$ and $x'_i = -1$.
    In the first case we obtain
    \begin{align*}
        2 \delta
        = \la c,x\ra + \la c,x'\ra
        & = [\la \pi_i(c), \pi_i(x)\ra + c_ix_i] + [\la \pi_i(c), \pi_i(x')\ra + c_ix'_i] \\
        & = [\la \pi_i(c), \pi_i(x)\ra + c_i] + [\la\pi_i(c), \pi_i(x')\ra] \\
        & = c_i.
    \end{align*}
    Similarly, in the second case we obtain $2 \delta = -c_i$.

    If $\delta = 0$, this would imply that $c = \zero$, a contradiction to the fact that $H \ne \R^d$.
    Otherwise, $\delta = 1$ and hence every nonzero coordinate of $c$ is $\pm 2$.
    Thus, for every $x \in \Z^d$ we see that $\la c,x\ra$ is an even number, in particular $\la c,x\ra \ne \delta$.
    This means that $X \subseteq \Z^d \cap H = \emptyset$, and we are done.
\end{proof}

\begin{lemma}
    \label{lemsliceb}
    Let $\aaa,\bb \subseteq \R^d$ such that $\aaa$ spans $\R^d$, $\bb$ does not contain opposite points, and for every $b \in \bb$ there is some $\eps_b \in \{ \pm 1\}$ such that $\{\la a,b\ra : a \in \aaa\} \subseteq \{0,\eps_b\}$.
    Then we have $|\bb| \le 2^{\dim \bb}$.
\end{lemma}
\begin{proof}
    Let $a_1,\dots,a_d \in \aaa$ be a basis of $\R^d$, and let $M \in \R^{d \times d}$ such that $a_i = M^\top e_i$ for $i=1,\dots,d$.
    For every $b \in \bb$ we obtain
    \[
        \la e_i, Mb\ra = %(M^{-\top} a_i)^\top Mb = 
        \la a_i, M^{-1}Mb\ra = \la a_i, b\ra \in \{0,\eps_b\}
    \]
    for $i=1,\dots,d$ and hence $Mb \in \{0,\eps_b\}^d$.
    Thus, the set $X := \{ Mb : b \in \bb\}$ is contained in $\{0,1\}^d \cup \{0,-1\}^d$.
    As $\bb$ does not contain opposite points, we see that also $X$ does not contain opposite points.
    Lemma~\ref{lemslice} now implies $|\bb| = |X| \le 2^{\dim X} = 2^{\dim \bb}$.
\end{proof}

We are ready to continue with the proofs of the remaining claims.

\claimpreimagespi*
\begin{proof}
    Let $y := \pi(b)$ for some $b \in \bb$ and observe that $\pi^{-1}(y) = \{x \in \R^d : \pi(x) = \pi(y)\}$ is a one-dimensional affine subspace.
    By~\eqref{eqscaleps} and Lemma~\ref{lemsliceb} we obtain $|\bb \cap \pi^{-1}(y)| \le 2$.
\end{proof}

\claimpreimagestau*
\begin{proof}
    Fix any $b \in \bb$ and let $v := \pi(b^*)$.
    Consider the orthogonal complement $W \subseteq U$ of $U_0$ in $U$.
    As $\tau^{-1}(\tau(v)) %\cap \pi(\bb) 
    = v + W %\cap \pi(\bb)
    $, it suffices to show that
    \[
        |(v + W) \cap \pi(\bb)| \le 2^{d - 1 - \dim U_0}
    \]
    holds.
    To this end, consider the linear subspace $\Pi \subseteq U$ spanned by $v$ and $W$ and let $\sigma : U \to \Pi$ denote the orthogonal projection on $\Pi$.
    
    First, suppose that $\sigma(\aaa_1')$ spans $\Pi$.
    For every $a \in \aaa_1' \subseteq U$ and every $b \in \bb$ with $\pi(b) \in v + W \subseteq \Pi $ we have
    \[
        \la \sigma(a), \pi(b)\ra = \la a,\pi(b)\ra = \la a,b\ra \in \{0,\gamma_b\}
    \]
    by~\eqref{eqscala1pgamma}.
    Moreover, recall that $\pi(\bb)$ does not contain opposite points by Claim~\ref{firstClaim}~(iii).
    Thus, the pair $\sigma(\aaa_1')$ and $(v + W) \cap \pi(\bb)$ satisfies the requirements of Lemma~\ref{lemsliceb} (in $U$), and hence we obtain
    \[
        |(v + W) \cap \pi(\bb)| \le 2^{\dim(v + W)} = 2^{\dim W} = 2^{\dim U - \dim U_0} = 2^{d - 1 - \dim U_0}.
    \]
    It remains to consider the case in which $\sigma(\aaa_1')$ does not span $\Pi$.
    Unless $|(v + W) \cap \pi(\bb)| = 1$, we will identify points $b_1,b_2 \in \bb$ with $\max \{ \varphi(b_1),\varphi(b_2) \} > \varphi(b_d)$, a contradiction to the choice of $b_d$.

    As $\aaa_0 \cup \aaa_1'$ spans $U$, we know that $\sigma(\aaa_0 \cup \aaa_1')$ spans $\Pi$.
    Since $\aaa_0$ is orthogonal to $W$, this means that $\sigma(\aaa_0)$ spans a line, and $\sigma(\aaa_1')$ spans a hyperplane $H$ in $\Pi$.
    Note that we have $v \notin W$ (otherwise $W = \Pi$ and so $\sigma(\aaa_1')$ spans $\Pi$).
    Thus, every nonzero point in $\sigma(\aaa_0)$ has nonzero scalar product with $v$.
    Moreover, for every $a \in \aaa_0$ with $\sigma(a) \ne \zero$ we have $\la \sigma(a), v\ra = \la a, v\ra = \la a,b\ra \in \{0,1\}$ by~\eqref{eqscala0}.
    Thus, since the nonzero vectors in $\sigma(\aaa_0)$ are collinear, we obtain
    \[
        \sigma(\aaa_0) \subseteq \{\zero, \sigma(a_0)\}
    \]
    for some $a_0 \in \aaa_0$. Since $\mathbf 0\in H,$ we have $\sigma(\aaa_0)\setminus H\subseteq \{\sigma(a_0)\}$ and further, since $\sigma(\aaa_0\cup \aaa_1')$ spans $\Pi$, we have $\sigma(\aaa_0)\setminus H= \{\sigma(a_0)\}$. 
    Let $c \in \Pi$ be a normal vector of $H$.  
    As $\sigma(a_0) \notin H$, we may scale $c$ so that $\la \sigma(a_0), c\ra = 1$.
    Let $a_* \in \aaa_1$ such that $\aaa_1' = \aaa_1 - a_*$.
    We define
    \[
        b_1 := c - \delta_1 b_d \ne \zero,
    \]
    where $\delta_1 := \la a_*, c\ra$.
    For every $a \in \aaa_0$ we have
    \[
        \la a,b_1\ra = \la a, c\ra = \la \sigma(a), c\ra \in \{\la \zero,c\ra, \la \sigma(a_0), c\ra\} = \{0,1\},
    \]
    and for every $a \in \aaa_1$ we have
    \begin{align*}
        \la a,b_1\ra
        = \la \underbrace{a - a_*}_{\in \aaa_1'}, b_1\ra + \la a_*, b_1\ra
        &= \la a - a_*, c\ra + \la a_*, b_1\ra
         = \la {\underbrace{\sigma(a - a_*)}_{\in H}}, c\ra + \la a_*, b_1\ra \\
        & = \la a_*, b_1\ra
        = \la a_*, c\ra - \delta_1 \la a_*, b_d\ra
        = \la a_*, c\ra - \delta_1
        = 0.
    \end{align*}
    Thus, by the maximality of $\bb$, (a scaling of) the vector $b_1$ is contained in $\bb$.
    Since we assumed $\zero \in \aaa_0$, we have $\varphi(b_1) \ge \dim(\aaa_1) + 1$.

    In order to construct $b_2$, let us suppose that there is another point $b' \in \bb$ with $v' := \pi(b') \ne v$ and $v' \in (v + W)$.
    If there is no such point, then the statement of the claim is true.
    Recall that $\sigma(a_0)$ is orthogonal to $W$, and let
    \[
        \xi := \la \sigma(a_0), v\ra = \la \sigma(a_0), \underbrace{v - v'}_{\in W}\ra + \la \sigma(a_0), v'\ra = \la \sigma(a_0), v'\ra.
    \]
    Choose $v'' \in \{v,v'\}$ such that $\xi c \ne v''$, and let $b'' \in \{b,b'\}$ such that $\pi(b'') = v''$.
    Define $\delta_2 := \la a_*,v'' - \xi c\ra$ and note that
    \[
        b_2 := v'' - \xi c - \delta_2 b_d
    \]
    is nonzero since $v'' - \xi c \in U \setminus \{\zero\}$.
    For every $a \in \aaa_0$ we have
    \[
        \la a,b_2\ra = \la a, \underbrace{v'' - \xi c}_{\in \Pi}\ra = \la \sigma(a), v'' - \xi c\ra,
    \]
    which is zero if $\sigma(a) = \zero$.
    Otherwise, $\sigma(a) = \sigma(a_0)$ and we obtain
    \[
        \la a,b_2\ra = \la \sigma(a_0), v''\ra - \xi \la \sigma(a_0), c\ra = \la \sigma(a_0), v''\ra - \xi = 0.
    \]
    Thus, $b_2$ is orthogonal to $\aaa_0$.
    Moreover, note that
    \[
        \la a_*, b_2\ra = \la a_*, v'' - \xi c\ra - \delta_2 \underbrace{\la a_*, b_d\ra}_{= 1} = 0.
    \]
    Thus, for every $a \in \aaa_1$ we have
    \begin{align*}
        \la a,b_2\ra
        = \la a - a_*, b_2\ra + \la a_*, b_2\ra
        = \la a - a_*, b_2\ra
        & = \la a - a_*, v''\ra - \xi \underbrace{\la a - a_*, c\ra}_{= 0} - \delta_2 \underbrace{\la a - a_*, b_d\ra}_{= 0} \\
        & = \la a - a_*, v''\ra = \la a - a_*, b''\ra \in \{0,\gamma_{b''}\}
    \end{align*}
    by~\eqref{eqscala1pgamma}.
    Thus, again by the maximality of $\bb$, (a scaling of) the vector $b_2$ is contained in $\bb$, and since $b_2$ is orthogonal to $\aaa_0$ and $a_* \in \aaa_1$, we have $\varphi(b_2) \ge \dim(\aaa_0) + 1$.
    However, by the choice of $b_d$ we must have
    \[
         \max \{ \dim(\aaa_0), \dim(\aaa_1) \} + 1 \le \max \{\varphi(b_1), \varphi(b_2)\} \le \varphi(b_d) = \max\{\dim(\aaa_0), \dim(\aaa_1)\},
    \]
    a contradiction.
\end{proof}

\claimrestbbconstant*
\begin{proof}
    Let $b \in \bb \setminus \bb_*$ and, for the sake of contradiction, suppose that $|\{ \la a,b\ra : a \in \aaa_0 \}| = |\{ \la a,b\ra : a \in \aaa_1 \}| = 2$.
    Let $b' \in \bb \setminus \{b\}$ such that $\pi(b) = \pi(b')$.
    In other words, we have $b' = b + \gamma b_d$ for some $\gamma \ne 0$.
    Then, by~\eqref{eqscala0} we have
    \[
        \{ \la a,b'\ra : a \in \aaa_0 \} = \{ \la a,b\ra : a \in \aaa_0 \} = \{0,1\}
    \]
    and hence we obtain $\eps_b = \eps_{b'} = 1$ by~\eqref{eqscaleps}.
    Again by~\eqref{eqscaleps} we see
    \[
        \{0,1\} \supseteq \{ \la a,b'\ra : a \in \aaa_1 \} = \{ \la a,b\ra : a \in \aaa_1 \} + \gamma = \{0,1\} + \gamma = \{\gamma, 1+\gamma\},
    \]
    which implies $\gamma = 0$, a contradiction.
\end{proof}

\claimaaaibbi*
\begin{proof}
    First, note that by~\eqref{eqscaleps} there is an invertible matrix $M \in \R^{d \times d}$ such that $M(\aaa) := \{ Ma : a \in \aaa\} \subseteq \{0,1\}^d$.
    Now let $i \in \{0,1\}$.
    Denote by $V$ the span of $\bb_i$ and let $k := \dim V$.
    Clearly, we have $\bb_i \subseteq \bb \cap V$ and hence by~\eqref{eqscaleps} and Lemma~\ref{lemsliceb} we obtain $|\bb_i| \le |\bb \cap V| \le 2^k$.
    Recall that for each $b \in \bb_i$ there some $\xi_b$ such that $\la a,b\ra = \xi_b$ holds for all $a \in \aaa_i$.
    Thus, $\aaa_i$ is a subset of $\aaa \cap V'$, where $V' := \{ x \in \R^d : \la x,b\ra = \xi_b \text{ for all } b \in \bb_i \}$, and hence we obtain
    \begin{align*}
        |\aaa_i| = |M(\aaa_i)| \le |M(\aaa \cap M(V'))| \le |\{0,1\}^d \cap M(V')| & \le 2^{\dim M(V')} \\
        & = 2^{\dim V'} \le 2^{d - \dim V} = 2^{d-k},
    \end{align*}
    where the third inequality follows from Lemma~\ref{lemslice}.
    We conclude that $|\aaa_i| |\bb_i| \le 2^{d-k} 2^k = 2^d$.  
\end{proof}

\section{Application to 2-level polytopes}
\label{sec2level}

In this section, we provide a proof for our main application:

\twolevelresult*

\begin{proof}
Let $P \subseteq \R^d$ be a $d$-dimensional 2-level polytope.
We may assume that $\zero$ is among the vertices of $P$.
Thus, there exists a finite set $\bb \subseteq \R^d \setminus \{\zero\}$ such that
\[
    P = \{ x \in \R^d : 0 \le \la x,b\ra \le 1 \text{ for every } b \in \bb \},
\]
where for each $b$, at least one of the equations $\la x,b\ra = 0$, $\la x,b\ra = 1$ defines a facet of $P$.
Let $\aaa$ consist of the vertices of $P$.
As $P$ is $d$-dimensional and pointed, both $\aaa$ and $\bb$ span $\R^d$ and hence $\aaa$ and $\bb$ satisfy the assumptions of Theorem~\ref{thmbinary}.
If no vector in $\bb$ defines two facets of $P$, we have
\[
    f_0(P) f_{d-1}(P) = |\aaa| |\bb| \le (d+1)2^d \le d2^{d+1}.
\]
If there exists a vector that defines two facets of $P$, then in the proof of Theorem~\ref{thmbinary} we may choose $b_d$ as this vector.
Indeed, using the notation of Section~\ref{secMainProof}, $\aaa_0$, $\aaa_1$ are then the vertex sets of the respective facets and $\varphi(b_d) = \dim \aaa_0 = \dim \aaa_1 = \dim U_0 = d-1$.
Recall the inequality~\eqref{eq8226}, which yields
\begin{align*}
    f_0(P) f_{d-1}(P) \le 2 |\aaa| |\bb| & \le (\dim U_0 + 1)2^{d+1} + 2|\aaa_0||\bb_0| + 2|\aaa_1||\bb_1| \\
    & = d2^{d+1} + 2|\aaa_0||\bb_0| + 2|\aaa_1||\bb_1|.
\end{align*}
We claim that $\bb_0 = \bb_1 = \emptyset$, in which case we are done.
To this end, suppose there is some $b \in \bb_i$.
Recall that $b$ has constant scalar product with all points in $\aaa_i$.
As $\bb_i \subseteq \bb$ and $\zero \notin \bb$, we obtain $b=b_d$.
However, using $\zero \notin \bb$ again, we have $b_d \in \bb_*$, contradicting the fact that $\bb_i \subseteq \bb \setminus \bb_*$.
\end{proof}

\bibliographystyle{abbrv}
\bibliography{references}

\end{document}